\documentclass[a4paper, 11pt]{article}
\usepackage{anysize}
\marginsize{3,5cm}{2,5cm}{2,5cm}{2,5cm}
\usepackage{amssymb}
\usepackage{amsmath,amsbsy,amssymb,amscd}
\usepackage{t1enc}
\usepackage[cp1250]{inputenc}
\usepackage[all]{xy}
\usepackage{color}
\usepackage{amsfonts}
\usepackage{latexsym}
\usepackage{amsthm}
\usepackage{mathrsfs}
\usepackage{hyperref}
\usepackage{graphicx}
\usepackage{comment}

\usepackage{color}
\usepackage{caption}
\usepackage{enumerate}
\usepackage{xcolor}
\usepackage[draft]{todonotes}

\definecolor{orange}{rgb}{1,0.5,0}

\usepackage{mathrsfs} 

\DeclareMathAlphabet{\mathpzc}{OT1}{pzc}{L}{it} 

\theoremstyle{definition}
\newtheorem{definition}{Definition}[section]
\newtheorem{theorem}[definition]{Theorem}

\newtheorem{proposition}[definition]{Proposition}

\newtheorem{lemma}[definition]{Lemma}

\newtheorem{remark}[definition]{Remark}

\newtheorem{claim}[definition]{Claim}

\def\geq{\geqslant}
\def\leq{\leqslant}

\def\epsilon{\varepsilon}

\def\loc{\operatorname{loc}}
\def\diam{\operatorname{diam}}

\newcommand{\bea}{\begin{eqnarray}}
  \newcommand{\eea}{\end{eqnarray}}
  \newcommand{\beab}{\begin{eqnarray*}}
  \newcommand{\eeab}{\end{eqnarray*}}

  \newcommand{\be}{\begin{equation}}
  \newcommand{\ee}{\end{equation}}

\newcommand{\abs}[1]{\left| #1 \right|}

\newcommand{\ve}{\epsilon}

\title{Horocycle flow on negative variable curvature surface is standard}
\author{Adam Kanigowski \and Kurt Vinhage\footnote{K. V. was supported by the National Science Foundation under Award DMS 1604796} \and Daren Wei\footnote{D. W. was partially supported by the NSF grant DMS-16-02409}}
\date{}

\begin{document}
\maketitle
\begin{abstract}
We provide a new proof that the horocycle flow preserving the Margulis measure on a variable negative curvature surface is standard\footnote{zero entropy loosely Bernoulli or loosely Kronecker}. This was first proved by Ratner in \cite{Ratner}. The main purpose of this note is to provide a simplified case of the arguments in \cite{KanigowskiVinhageWei}, which have similar but more complicated structures in the case of homogeneous flows on Lie groups, as well as illustrate the versatility of the method by applying it to a non-homogeneous flow.
\end{abstract}
\section{Introduction}
Let $M$ be a $C^3$ compact orientable surface of negative curvature and denote $U(M)$ as its unit tangent bundle. $U(M)$ carries a the natural geodesic flow , $g_s$. It is well known that 
$g$ is a $C^2$ anosov flow. As a result, we know for geodesic flow on $U(M)$, there exist invariant, orientable 1-dimensional foliations $\mathscr{H}$ and $\mathscr{K}$ of expanding curves and contracting curves, respectively. Since these foliations are 1-dimensional, it is natural to try to construct flows for which the leaves of $\mathscr{H}$ are orbits. There are many ways to do this, but two of them are the most natural: one may define $h_s(x)$ to be the point $y$ at distance $s$ from $x$ along the leaf $\mathscr{H}(x)$ (there are two such points, but the orientation on $\mathscr{H}(x)$ fixes the choice). Another is to do a similar process but with respect to a canonical measure on the leaf $\mathscr{H}(x)$, the disintegration of the {\it Margulis measure}.

We consider the second case in this paper, for a precise construction of the flow and its properties, see Section \ref{sec:para}. Ergodic properties of such flows were studied by Marcus, Feldman and Ornstein, and Ratner, among others (see \cite{Marcus,FelOrn,Ratner}). Marcus showed in \cite{Marcus} that such flows are uniquely ergodic. In \cite{Ratner}, Ratner showed that these flows (in fact any such flow parameterizing the leaves of $\mathscr{H}(x)$) were {\it standard}, also called {\it zero entropy loosely Bernoulli}. We provide a new proof of this result using different methods:

\begin{theorem}\label{thm:Main}
The horocycle flow $h_t$ preserving the Margulis measure is standard for any variable negative curvature surface.
\end{theorem}

The methods used here resemble those of a paper by the same authors \cite{KanigowskiVinhageWei}, and this note may serve as a good starting point before the more technical and involved arguments found there.

\section{Preliminaries on Standard Systems}
Standardness (zero entropy loosely Bernoulli or loosely Kronecker) is a concept introduced by A. Katok \cite{Katok1} and J. Feldman \cite{Feldman}. Standardness can be thought of as a measure of how closely a system resembles to a rigid action, as any standard system can be represented as a special flow over irrational rotation.







The notion of a standard system has many different equivalent definitions., as was shown in \cite{Katok1,Feldman,RatnerKaku1}. For our purposes, the most convenient is that of \cite{RatnerKaku1}, we summarize the notions and definition here.


The basic setting is following: suppose $h_t$ is an ergodic flow on a Lebesgue space $(X,\mathscr{B},\mu)$ and $\mathscr{P}$ is a finite measurable partition of $X$. For $x\in X$, denote $\mathscr{P}(x)$ as the atom of $\mathscr{P}$ which contains $x$ and define $I_R(x)=\{T_sx:s\in[0,R]\}$ for $R\geq0$. Let $l$ as the Lebesgue measure on $[0,R]$.

\begin{definition}[$(\epsilon,\mathscr{P})-$matchable, \cite{RatnerKaku1}]
For $x,y\in X$, $\epsilon>0$ and $R>1$, $I_R(x)$ and $I_R(y)$ are called {\it $(\epsilon,\mathscr{P})-$matchable} if there exists a subset $A=A(x,y)\subset[0,R]$, $l(A)>(1-\epsilon)R$ and an increasing absolutely continuous map $\psi=\psi(x,y)$ from $A$ onto $A'=A'(x,y)\subset[0,R]$, $l(A')>(1-\epsilon)R$ such that $\mathscr{P}(T_tx)=\mathscr{P}(T_{h(t)}y)$ for all $t\in A$ and derivative $\psi'=\psi'(x,y)$ satisfies
$$|\psi'(t)-1|<\epsilon\text{ for all }t\in A.$$
we call $\psi$ an $(\epsilon,\mathscr{P})-$matching from $I_R(x)$ onto $I_R(y)$.
\end{definition}


Let 
$$f_R(x,y,\mathscr{P})=\inf\{\epsilon>0:\text{ $I_R(x)$ and $I_R(y)$ are $(\epsilon,\mathscr{P})-$matchable}\}.$$

While $f_R$ is {\it not} a metric, it has enough similar properties to justify constructing $f_R$-balls.. Let $B_R(x,\epsilon,\mathscr{P})=\{y\in X:f_R(x,y,\mathscr{P})<\epsilon\}$ be the $(R,\mathscr{P})-$ball of radius $\epsilon>0$ centered at $x\in X$, $R>1$. $B_R(x,\ve,\mathscr{P})$ should be thought of as analogous to a Hamming ball in entropy theory, with the exception that we allow for a small time change given by $\psi$. 



\begin{theorem}[\cite{RatnerKaku1},\cite{RatnerKaku2}]
\label{thm:standard-criterion}
A zero-entropy ergodic measure preserving flow $h_t$ is standard if and only if for every $\ve > 0$ and a generating family of partitions $\mathscr{P}_i$, there exists $N(\ve,\mathscr{P}_i) > 0$ such that for every $R > 0$, there exist $x_1,\dots,x_n$ with $n \le N(\ve,\mathscr{P}_i)$ such that:

\[ \mu\left( \bigcup_{k=1}^n B_R(x,\ve,\mathscr{P}_i) \right) > 1-\ve \]
\end{theorem}

Therefore, the standardness of the system can be detected by computing the decay rate (or lack thereof) of $B_R(x,\ve,\mathscr{P})$ in $R$. In fact, because if $f_R(x,y,\mathscr{P}), f_R(y,z,\mathscr{P}) < \ve$, $f_R(x,z,\mathscr{P}) < 5\ve$, we know that a minimal cover of a set of size $1-\ve$ by $\ve$-$(R,\mathscr{P})$ balls will have their $\ve/5$-$(R,\mathscr{P})$ balls disjoint. In particular, if $\mu(B_R(x,\ve,\mathscr{P}))$ is bounded below independently of $x$ and $R$, we may conclude that $h_t$ is standard.



\section{Preliminaries on $W^u$ flows}\label{sec:para}


Let $S$ be a compact, negatively curved, oriented surface, and $g_s$ be the corresponding goedesic flow on its unit tangent bundle, $M$. There exists a 1-dimensional unstable foliation, with smooth leaves $W^u(x)$ for any $x\in M$. Since $S$ is oriented, the leaves $W^u(x)$ can be given an orientation. We wish to define a continuous 
flow $h_t$ whose orbits are exactly $W^u(x)$, but such a flow will depend on the way we parameterize each leaf.


We will introduce a very special parameterization by first constructing the desired invariant measure:

\subsection{Margulis measure}
\label{subsec:margulis}
The Margulis measure for a transitive Anosov flow was introduced by G. A. Margulis in \cite{Margulis}. For any $p\in M$, let  $W^*(p)$ be the leaf of the stable or unstable foliation  through $p$ for $* = s,u$, respectively, and $W^{0*}(p)$ be the leaf of the center-stable or center-unstable foliation foliation  through $p$ for $* = s,u$, respectively. Recall that $W^{0s}$ is the joint integration of the orbits of $g^t$ and the $W^s$ foliation, and similarly for $W^{0u}$. Define $OC(W^{0u}(p))$ as the collection of open sets in $W^{0u}(p)$ with compact closure and $$OC(W^{0u})=\bigcup_{p\in M}OC(W^{0u}(p)).$$

Let $O$ be a subset of $M$ contained in a sufficiently small neighborhood of a point $p$. Then set $f_O(x)=\mu^s((\{x\}\times U^s(p))\cap O)$ where $x\in U^{0u}(p)$ and $U(p)$ is a neighborhood of $p$ with local product form $U(p)=U^{0u}(p)\times U^s(p)$ such that $U^{0u}(p)\subset W^{0u}(p)$ and $U^s(p)\subset W^s(p)$. Then the Margulis measure will be a finite $g^t-$invariant measure with the following form:

\begin{proposition}[Margulis Measure, \cite{Margulis}]
\label{prop:margulis-measure}
There exists a unique family of measures $\mu^{0u}$ defined on the leaves $W^{0u}(x)$, $x \in M$, a unique family of measures $\mu^s$ defined on the leaves $W^s(x)$, $x \in M$, and a unique $g^t$-invariant measure $\mu$ on $M$ such that:

\begin{enumerate}[(a)]
\item $\mu^{0u}$ has the uniformly expansion property:
\begin{equation}
\mu^{0u}(g^t(U))=e^{h^ut}\mu^{0u}(U)
\end{equation}
where $U\subset OC(W^{0u})$, $h^u>0$ and $t\in\mathbb{R}$.
\item $\mu^s$ has the uniformly contraction property:
\begin{equation}
\mu^s(g^t(U'))=e^{h^st}\mu^s(U')
\end{equation}
where $U'\subset OC(W^s)$, $h^s<0$ and $t\in\mathbb{R}$.
\item Letting $\mu_q(A)=\mu^{0u}(A\times\{q\})$ for $q\in U^s(p)$, $A\subset U^{0u}(p)$, we have

\begin{equation}\label{eq:MargulisMeasure}
\mu(O)=\int f_O(x)d\mu_q(x).
\end{equation}
\end{enumerate}

Furthermore, $\mu$ is the unique measure of maximal entropy for $g^t$.
\end{proposition}
\begin{remark}
\label{rem:reversible}
Because $\mu$ is the unique measure of maximal entropy for $g^t$, it is also the unique measure of maximal entropy for $g^{-t}$, another Anosov flow. Therefore, $\mu$ disintegrates in the analogous way for $\mu^{0s}$ and $\mu^u$.
\end{remark}

We now fix the flow $h_t$ whose orbits are the leaves $W^u(x)$: if $x \in M$, let $h_t(x)$ be the point $y$ such that the interval connecting $x$ and $y$ has $\mu^u$-measure exactly $t$. Of course, there are two such points, but the leaves $W^u(x)$ are oriented, so if $t > 0$, we choose $y$ in the positive direction, and if $t < 0$, we choose $y$ in the negative direction.

Note that in most cases, this action is only H\"older, and is not even generated by a vector field. The ergodic properties of these flows were first studied by B. Marcus in \cite{Marcus}, and later by Feldman and Ornstein in \cite{FelOrn}. We remark that making this choice for $h_t$ yields that

\[ g_sh_t = h_{e^st}g_s \]

by Proposition \ref{prop:margulis-measure}(b) and Remark \ref{rem:reversible}. Notice also that by a completely symmetric argument, an analogous flow $k_t$ may be built whose orbits are the leaves of $W^s$, and

\[ g_s k_t = h_{e^{-s}t}g_s. \]

Finally, because of the product structure provided by Proposition \ref{prop:margulis-measure}(c), the map:

\[f : (r,s,t) \mapsto g_rh_sk_t(x) \]

is a homeomorphism onto a neighborhood of $x$, and $f_*(dr \times ds \times dt) = \mu$.


\subsection{Estimates relating the flows $h_t$, $k_t$ and $g_t$}
In this section, we recall some lemmas from \cite{FelOrn} which will be helpful for the proof of the main theorem. We remark that for clarity, we h use $h_t$ and $k_t$ to denote the flows coming from the Margulis parameterization as described in the previous subsection, while the notation of \cite{FelOrn} uses $\bar{h}_t$ and $\bar{k_t}$ to denote these flows.

Since Anosov flows have a local product structure for their weak-stable and unstable manifolds, if $y \in W^u(x)$ and $z \in W^{0s}(x)$ are sufficiently close to $x$, there is a unique point in $W^{0s}_{\loc}(y) \cap W^u_{\loc}(z)$. Hence, there are unique continuous functions $\rho(x,s,t),\sigma(x,s,t),\tau(x,s,t)$ such that
\begin{equation}
g_{\rho}k_{\tau}h_sx=h_{\sigma}k_tx,
\end{equation}

$\sigma$ and $\tau$ has the same sign as $s,t$, respectively, and $\rho(x,0,t) = \rho(x,s,0) = 0$. Furthermore, $\rho(x,s,t)$ enjoys the equivariance property:

\begin{equation}
\label{eq:rho-equivariance}
\rho(x,e^{-u} s,e^ut) = \rho(g_ux,s,t)
\end{equation}

In particular, we may define $\rho$ not only for sufficiently small $s,t$, but for all $s,t$ such that $\abs{st}$ is sufficiently small. Similarly, we get definitions on such $s,t$ for $\sigma,\tau$. We have the following results of \cite{FelOrn}, the first of which appeared as Lemmas 3.4 and 2.8, respectively:




\begin{lemma}\label{lemma:3.4}
The function $\sigma$ is differentiable in the second argument, and for sufficiently small $st$,
$$\frac{\partial\sigma}{\partial s}(x,s,t)=e^{\rho(x,s,t)}.$$
\end{lemma}

\begin{lemma}\label{lemma:3.8}
Given $\epsilon>0$, $\alpha$ may be chosen so small that if $y=g_uh_vk_wx$ with $|u|,|v|,|w|<\alpha$ and $|s|<\frac{\alpha^2}{2\abs{w}}$, then $d(h_sx,h_{e^u\sigma(x,s,w)}y)<\epsilon$.
\end{lemma}

\subsection{$u$-partitions}
Recall that when we consider the standardness, it is very crucial to consider the a suitable generating partition. In this paper, we will use the idea of $u-$partition which was introduced in \cite{Ratner}.

\begin{definition}[$u-$isomorphic and $s-$isomorphic \cite{Ratner}]
Two sets $A,B\subset W^{0s}$ are called $u-$isomorphic if there is continuous $\psi:A\times I\to M$ s.t.
\begin{itemize}
\item $\psi(x,I)\subset W^u$, $x\in A$;
\item $\psi(x,0)=x$, $\psi(x,1)\in B$ and the map $\tilde{\psi}:A\to B$, $\tilde{\psi}(x)=\psi(x,1)$ is a homeomorphism.
\end{itemize}
The set $\psi(A\times I)=P$ is called a $u-$cylinder with faces $A,B$. If the positive direction on orbits of $h_t$ goes from $A$ to $B$ we write $A=A_1(P)$, $B=A_2(P)$. $\psi(x,I)$ and $\psi(y,I)$, $x,y\in A$ are called $s-$isomorphic.
\end{definition}

\begin{definition}[$u-$partition \cite{Ratner}]
Let $\alpha=\{P_1,\ldots,P_m\}$ be a partition of $M$ into $u-$cylinders and we will say $\alpha$ is a $u-$partition.
\end{definition}

Recall that Margulis measure has local product structure, thus the boundary of this partition has measure zero. Another important feature of this partition is that we can formulate a family of generating partitions based on this partition. In fact, by shrinking the size of $A,B$ and distance of $A,B$, it is clear that we can construct a family of generating partitions:

\begin{lemma}[Ratner, \cite{Ratner}]
For every $\delta > 0$, there exists a $u$-partition of $M$ such that $\diam(P_i) < \delta$
\end{lemma}

\section{Proof of Theorem \ref{thm:Main}}\label{sec:uniformlyExpanding}
In this section, we will prove the Theorem \ref{thm:Main}. Before we give a proof of this theorem, we will introduce some definitions and tools which will help us formulate the proof.

From now, we will prove the Theorem \ref{thm:Main} for a fixed $u-$partition. And then the same proof can be applied to any other partitions and thus we done.

Fix a $u-$partition $\alpha_m$, let $V_{\epsilon^2}^m$ be the $\epsilon^2$ neighborhood of the boundary of $\alpha_m$. Since Margulis measure $\mu$ has a local product measure on $W^{0s}$ and $W^u$, it follows that $\mu(V_{\epsilon^2}^m)=O(m^2\epsilon^2)$. By ergodic theorem on $h_t$ and $\chi_{V_{\epsilon^2}^m}$, we obtain a set $D_{\epsilon}$ such that $\mu(D_{\epsilon})>1-\epsilon^2$ and a number $N_{\epsilon}>0$ such that for every $R\geq N_{\epsilon}$ and every $y\in D_{\epsilon}$, we have
\begin{equation}
|\{t\in[0,R]:h_t(y)\in V_{\epsilon^2}^m\}|\leq\frac{\epsilon R}{2}.
\end{equation}

If we can prove the upper bound of the number of Kakutani matching balls is bounded, then we finish our proof. Indeed, the upper bound of number of Kakutani balls will follow from the following three claims:
\begin{definition}[Pre-Matching Balls]\label{def:kakball}
If $x,y\in M$ we say that
$x\in PM(R,\epsilon,y)$ if and only if $x = g_uh_vk_wy$ and
$\abs{w} < \ve / R$, $\abs{u} < \ve$ and $\abs{v}<\epsilon$.
\end{definition}

Notice that since the coordinates $(u,v,w) \mapsto g_uh_vk_wy$ take Lebesgue measure to the Margulis measure, $\mu(PM(R,\ve,y)) = \ve^3/R$ (recall the end of Section \ref{subsec:margulis}).

\begin{claim}\label{ClaimA}
For every $y\in D_{\epsilon}$, every $R\geq N_{\epsilon}$ and every $x\in M$ of there exists $p\in[0,\epsilon^3R]$ such that if for $|u|,|v|<\epsilon^5$ and $|w|<\frac{\epsilon^5}{R}$:
$$h_px\in PM(R,\epsilon^5,y)$$
Then $x\in B_R(y,\epsilon,\alpha_m)$.
\end{claim}
\begin{claim}\label{ClaimB}
For every $y\in D_{\epsilon}$ and every $p,q\in[0,\epsilon^3R]$, with $|p-q|\geq1$, we have
$$h_{-p}(PM(R,\epsilon^5,y))\cap h_{-q}(PM(R,\epsilon^5,y))=\emptyset.$$
\end{claim}

Let us prove Theorem \ref{thm:Main} before showing each claim:

\begin{proof}[Proof of Theorem \ref{thm:Main}]
Suppose $y\in D_{\epsilon}$. Then by claim \ref{ClaimA}, we have
$$\bigcup_{p\in[0,\epsilon^3R]}h_{-p}(PM(R,\epsilon^5,y))\subset B_R(y,\epsilon,\alpha_m).$$

By claim \ref{ClaimB} and the formula for $\mu(PM(R,\ve^5,y))$, we have
$$\mu(B_R(y,\epsilon,\alpha_m))\geq \mu(\bigcup_{p\in[0,\epsilon^3R]}h_{-p}(PM(R,\epsilon^5,y)))\geq\epsilon^3R\mu(PM(R,\epsilon^5,y))\geq \epsilon^{18}.$$

By the remarks following the statement of Theorem \ref{thm:standard-criterion}, we conclude $h_t$ is standard.
\end{proof}

\begin{proof}[Proof of Claim \ref{ClaimA}]
Take $y\in D_{\epsilon}$ and let $p\in[0,\epsilon^3R]$ be such that $h_px\in PM(R,\epsilon^5,y)$.

By Definition \ref{def:kakball}, we have for some $|u|,|v|<\epsilon^5$, $|w|<\frac{\epsilon^{\color{red} 5}}{R}$ such that
\begin{equation}
x= h_{-p}g_uh_vk_wy.
\end{equation}

Let $\psi(t)=e^u\sigma(x,t,w)$ as Lemma \ref{lemma:3.8}, $l(t)=\psi(t)+p$, $A(x,y)=\{t\in[0,R]: h_t(y)\notin V_{\epsilon^2}^m\}$. Moreover, by Lemma \ref{lemma:3.4} and \eqref{eq:rho-equivariance}:

\[ l'(t) = e^u\dfrac{\partial \sigma}{\partial t} (x,t,w) = e^{u + \rho(x,t,w)} = \exp(u + \rho(g_{-\log t}x,\ve,wt/\ve)). \]

Notice that $u$ is small and $\abs{w} < \ve^5 / R < \ve^5/t$, so $\abs{wt/\ve} < \ve^4$. Therefore, since $\rho(x,\ve,0) = 0$ for all $x$ and $(x,t) \mapsto \rho(x,\ve,t)$ is jointly continuous, we have, for sufficiently small $\ve$ depending on the fixed function $\rho$, $$|l'(t)-1|<\epsilon$$ for every $t\in[0,R]$ and hence $l$ satisfies the matching function condition of $(\epsilon,\alpha_m)$.

By Lemma \ref{lemma:3.8}, we know that
\begin{equation}
d(h_ty,h_{l(t)}x)\leq \epsilon.
\end{equation}

Thus for every $t\in A(x,y)$, we have $\alpha_m(h_ty)=\alpha_m(h_{l(t)}x)$. Since $|A(x,y)|\geq (1-\epsilon)R$, thus we have $x\in B_R(y,\epsilon,\alpha_m)$ and finish the proof.
\end{proof}

\begin{proof}[Proof of Claim \ref{ClaimB}:]
Suppose that there exists $x\in h_{q-p}(PM(R,\epsilon^5,y))\cap PM(R,\epsilon^5,y)$ with $\epsilon^3R\geq|p-q|\geq1$ and $y\in D_{\epsilon}$. Thus by the Definition \ref{def:kakball} and denote $r=p-q$, we have
$$x=g_uh_vk_wy$$
and
$$h_{-r}x=g_{u'}h_{v'}k_{w'}y$$
where $|u|,|u'|,|v|,|v'|\leq\epsilon^5$ and $|w|,|w'|\leq\frac{\epsilon^{\color{red} 5}}{R}$.

Therefore using the first equality to express $x$ in the second equality, we have,
$$h_{-r}g_uh_vk_wy=g_{u'}h_{v'}k_{w'}y$$

Recall that $h_{-r}g_u=g_uh_{-e^{-u}r}$ thus we have
\begin{equation}\label{eq:ClaimBcontradict}
g_uh_{v-e^{-u}r}k_wy=g_{u'}h_{v'}k_{w'}y.
\end{equation}

Then acting by $k_{-w'}h_{-v'}g_{-u'}$ on the both sides of \eqref{eq:ClaimBcontradict}, we have
\begin{equation}\label{eq:ClaimBcontradict2}
k_{-w'}h_{-v'}g_{u-u'}h_{v-e^{-u}r}k_wy=y.
\end{equation}

We now further acti on both sides of \eqref{eq:ClaimBcontradict2} by $g_{-t}$ and pushing it through to $y$ via conjugation. We select $t=\log(\epsilon^{-2} r)$, to get
\begin{equation}
(g_{-t}k_{-w'}g_t)(g_{-t}h_{-v'}g_t)g_{u-u'}(g_{-t}h_{v-e^{-u}r}g_t)(g_{-t}k_wg_t)g_{-t}y=g_{-t}y.
\end{equation}
By renormalization relation, we have,
\begin{equation}\label{eq:ClaimBcontradict3}
k_{-e^tw'}h_{-e^{-t}v'}g_{u-u'}h_{e^{-t}(v-e^{-u}r)}k_{e^tw}g_{-t}y=g_{-t}y.
\end{equation}

Since $\max\{|e^tw'|,|e^{-t}v'|,|u-u'|,|e^tw|\}\leq\epsilon^5$ and $2\epsilon\geq|e^{-t}(v-e^{-u}r)|\geq\frac{1}{2}\epsilon^2$  then we have the left side of \eqref{eq:ClaimBcontradict3} is $\epsilon^2$ away from $g_{-t}y$ and thus we get a contradiction.
\end{proof}



\end{document}